\documentclass[12pt,twoside]{article}
\title{$n$-gr-Coherent rings and Gorenstein graded modules}
\date{}
\author{}
\usepackage{graphicx,fancyhdr,fancybox,rotating,xcolor,appendix}
\usepackage{ulem,hhline,dsfont,pifont,multicol,lmodern}
\usepackage{amsthm,amsbsy,geometry,times,pst-node}
\usepackage{amsmath,tikz-cd}

\usepackage{lipsum}
\usepackage{euscript}
\usepackage{hyperref}
\usepackage{hyperref}
\usepackage{hyperref}

\usepackage{textcomp}
\usepackage{pdfpages}
\usepackage{stmaryrd}
\usepackage[all]{xy}
\usepackage{amsfonts}
\usepackage{amsmath}
\usepackage{amssymb}
\usepackage{latexsym}
\usepackage{mathrsfs}

\newtheorem{thm}{Theorem}[section]
 \newtheorem{cor}[thm]{Corollary}
 \newtheorem{lem}[thm]{Lemma}
 \newtheorem{prop}[thm]{Proposition}
 \newtheorem{Def}[thm]{Definition}
\newtheorem{rem}[thm]{Remark}
 \newtheorem{ex}[thm]{Example}

\newcommand{\X}{\rm \mathscr{X}}
\newcommand{\Y}{\rm \mathscr{Y}}

\def\Ext{{\rm Ext}}
\def\Tor{{\rm Tor}}
\def\Hom{{\rm Hom}}

\def\Ker{{\rm Ker}}

\newcommand{\F}{{\cal F}}
\newcommand{\M}{{\cal M}}

\def\pd{{\rm pd}}
\def\fd{{\rm fd}}
\def\id{{\rm id}}

\begin{document}

\thispagestyle{empty}

\maketitle \vspace*{-1.5cm}
\begin{center}{\large\bf  Mostafa Amini$^{1,a}$, Driss Bennis$^{2,b}$ and Soumia Mamdouhi$^{2,c}$}
\bigskip

\small{1. Department of Mathematics, Faculty of Sciences, Payame Noor University, Tehran, Iran.\\ 
2. Department of Mathematics, Faculty of Sciences, Mohammed V University in Rabat, Rabat, Morocco.\\
$\mathbf{a.}$ amini.pnu1356@gmail.com\\
$\mathbf{b.}$ driss.bennis@um5.ac.ma; driss$\_$bennis@hotmail.com\\
$\mathbf{c.}$ soumiamamdouhi@yahoo.fr
}

\end{center}

\bigskip

\noindent{\large\bf Abstract.} Let $R$ be a graded ring and $n\geq1$ an integer. In this paper, we introduce and study  the notions of Gorenstein $n$-FP-gr-injective and Gorenstein $n$-gr-flat modules  by using the notion of special finitely presented graded modules.   On $n$-gr-coherent rings, we investigate
the relationships between Gorenstein $n$-FP-gr-injective and Gorenstein $n$-gr-flat modules. Among other results,
we prove that any graded module in $R$-gr (resp. gr-$R$) admits Gorenstein $n$-FP-gr-injective (resp. Gorenstein $n$-gr-flat) cover and preenvelope.
\bigskip

\medskip

\bigskip\bigskip
%

\section{Introduction}

\ \ In 1990s, Enochs, Jenda and
Torrecillas, introduced the concepts
of Gorenstein injective and Gorenstein flat modules over arbitrary
rings \cite{EJ, EO}.  In 2008, Mao
and Ding introduced a special case of the Gorenstein
injective modules and they called Gorenstein FP-injective modules, which  renamed by Gillespie  by Ding injective \cite{JG}.  These Gorenstein FP-injective modules are stronger than the Gorenstein injective modules, and  in general an FP-injective module is not necessarily
Gorenstein FP-injective \cite[ Proposition 2.7]{L.X}. For this reason, Gao and Wang introduced and studied in \cite{Z.G} another  notion called Gorenstein FP-injective
modules which is  weaker than the usual Gorenstein injective modules. Furthermore,  all FP-injective modules are  in the class of Gorenstein FP-injective modules (see Section 2 for the definitions of these notions).

In this paper we deal with the graded aspect of some   extensions of these notions.  As it is known,  graded rings and modules are a classical notions in algebra  which   build  their  values and strengths from their connection with   algebraic geometry (see for instance \cite{CNA, CNS, CNF}).   Several authors have investegated the graded aspect of some notions in relative homological algerbra. For example,  Asensio,  L\'{o}pez Ramos and   Torrecillas in \cite{MJO, MJL} introduced the notions of Gorenstein gr-projective, gr-injective and gr-flat modules. In the recent years, the Gorenstein homological theory for graded rings have become an important area of research (see for instance \cite{MJM, Z.P}).  The  notions of FP-gr-injective  modules was introduced  in \cite{MJP}, and in  \cite{ZL} homological behavior of the FP-gr-injective  modules on gr-coherent rings were investigated. Along the same lines, it is natural to generalize the notion of  ``FP-gr-injective  modules and gr-flat modules'' to ``$n$-FP-gr-injective modules and $n$-gr-flat modules''. This done by Zhao, Gao and Huang in \cite{NG} basing on the notion  of   special finitely
presented graded modules which they defined  via projective resolutions of $n$-presented graded modules. Recently,
 in 2017,  Mao via FP-gr-injective modules gave a definition of Ding gr-injective modules \cite{L.X2}. Under this definition these Ding gr-injective modules are stronger than the Gorenstein gr-injective modules, and an FP-gr-injective module is not necessarily Ding gr-injective in general \cite[ Corollary 3.7]{L.X2}. So, for any $n\geq1$,  we study the consequences of extending the notion of $n$-FP-gr-injective and $n$-gr-flat modules to that of Gorenstein $n$-FP-gr-injective and Gorenstein $n$-gr-flat modules, respectively. Then, in this paper, for any $n\geq1$ by using $n$-FP-gr-injective modules and $n$-gr-flat modules, we introduce  a concept of Gorenstein $n$-FP-gr-injective  and Gorenstein $n$-gr-flat modules, and under this definition,  Gorenstein $n$-FP-gr-injective and Gorenstein $n$-gr-flat modules are weaker than the usual Gorenstein gr-injective and Gorenstein gr-flat modules, respectively. Also, for any $n\geq 1$, all gr-injective, $n$-FP-gr-injective modules and gr-flat, $n$-gr-flat modules are Gorenstein $n$-FP-gr-injective and Gorenstein $n$-gr-flat, respectively, and in general, Gorenstein $n$-FP-gr-injective and Gorenstein $n$-gr-flat $R$-modules need not be $n$-FP-gr-injective and $n$-gr-flat, unless in certain cases, see Proposition \ref{2.2a}.

The paper is organized as follows:

In Sec. 2, some fundamental concepts and some preliminary results are stated.

In Sec. 3, we introduce Gorenstein $n$-FP-gr-injective and  Gorenstein $n$-gr-flat modules for an integer $n\geq 1$ and then we give some characterizations of these modules.  Among other results, we prove that, for an exact sequence $0\rightarrow A \rightarrow B \rightarrow C\rightarrow 0$ of graded left $R$-modules, if $A$ and $B$ are Gorenstein $n$-FP-gr-injective, then $C$ is Gorenstein $n$-FP-gr-injective if and only if every $n$-presented module in $R$-gr with gr-${\rm pd}_{R}(U)<\infty$ is $(n+1)$-presented, and it follows that $(^{\bot}\mathcal{G}_{gr-\mathcal{FI}_n}, \mathcal{G}_{gr-\mathcal{FI}_n})$ is a hereditary cotorsion pair if and only if every $n$-presented module in $R$-gr with gr-${\rm pd}_{R}(U)<\infty$ is $(n+1)$-presented and every $M\in(^{\bot}\mathcal{G}_{gr-\mathcal{FI}_n})^{\bot}$ has an exact left ($gr$-$\mathcal{FI}_n$)-resolution, where $\mathcal{G}_{gr-\mathcal{FI}_n}$  and $gr$-$\mathcal{FI}_n$ denote the classes of Gorenstein $n$-FP-gr-injective and $n$-FP-gr-injective modules in $R$-gr,
 respectively.
 Also,  for a graded left (resp. right) $R$-module $M$ over a left $n$-gr-coherent ring $R$: $M$ is Gorenstein $n$-FP-gr-injective (resp. Gorenstein $n$-gr-flat) if and only if $M^{*}$ is Gorenstein $n$-gr-flat (resp. Gorenstein $n$-FP-gr-injective). Furthermore, the class of Gorenstein $n$-FP-gr-injective (resp. Gorenstein $n$-gr-flat) modules are  closed under direct limits (resp. direct products).  In this section, examples are given in order to show that Gorenstein $m$-FP-gr-injectivity (resp. Gorenstein $m$-gr-flatness) does not imply Gorenstein $n$-FP-gr-injectivity (resp. Gorenstein $n$-gr-flatness) for any $m>n$. Also, examples are given   showing that Gorenstein $n$-FP-gr-injectivity does not imply gr-injectivity.  In this paper, $gr$-$\mathcal{I}$ denote the classes of gr-injective modules in $R$-gr and  $gr$-$\mathcal{F}$, $gr$-$\mathcal{F}_n$ and $\mathcal{G}_{gr-\mathcal{F}_n}$ denote the classes
of gr-flat, $n$-gr-flat and  Gorenstein $n$-gr-flat  modules in gr-$R$, respectively.

In Sec. 4, it is shown that the class of Gorenstein $n$-FP-gr-injective and  Gorenstein $n$-gr-flat modules are covering and  preenveloping on $n$-gr-coherent rings.  We also establish some equivalent characterizations of
$n$-gr-coherent rings in terms of Gorenstein $n$-FP-gr-injective and Gorenstein $n$-gr-flat modules.


\section{Preliminaries}
\ \ Throughout this paper, all rings considered are associative with identity element
and the $R$-modules are unital. By $R$-Mod and Mod-$R$ we will denote the    category  of all left $R$-modules and  right $R$-modules, respectively.

In this section, some fundamental concepts and notations are stated.
 
  Let $n$ be a non-negative integer and
$M$ a left $R$-module. Then,
 $M$  is said to be {\it Gorenstein injective} (resp. {\it Gorenstein flat})
\cite{ EJ, EO} if there
is an exact sequence $$
\cdots\longrightarrow I_1  \longrightarrow I_{0} \longrightarrow I^0\longrightarrow I^1\longrightarrow\cdots$$ of injective (resp. flat) left $R$-modules
with $M= {\rm ker}(I^0\rightarrow
I^1)$ such that ${\Hom}_{R}(U,-)$ (resp.  $U\otimes_{R}-$) leaves the sequence exact whenever
$U$ is an injective left (resp. right) $R$-module.

$M$ is said to be {\it $n$-presented} \cite{Costa, DKM} if there is an exact sequence
$$
F_n\longrightarrow F_{n-1} \longrightarrow\cdots \longrightarrow F_1\longrightarrow F_0\longrightarrow U\longrightarrow 0$$ of left $R$-modules, where each $F_i$ is finitely generated free, and a ring $R$ is called left {\it $n$-coherent}, if every
$n$-presented left $R$-module is $(n + 1)$-presented.
$M$ is said to be {\it $n$-FP-injective} \cite{CD} if ${\Ext}_{R}^{n}(U,M)=0$
 for any $n$-presented left $R$-module $U$. In case $n=1$, $n$-FP-injective modules are nothing but the well-known  FP-injective modules. A right module
$N$ is called {\it $n$-flat} if ${\rm Tor}_{n}^{R}(N,U)=0$
 for any $n$-presented left $R$-module $U$.
 
$M$ is said to be {\it Gorenstein FP-injective} \cite{L.X} if there is an exact sequence $$
{\mathbf{E}}=\cdots\longrightarrow E_1 \longrightarrow E_{0} \longrightarrow E^0\longrightarrow E^1 \longrightarrow\cdots$$ of injective  left modules with $M=\ker(E^0\rightarrow E^1)$ such that ${\Hom}_{R}(U,{\mathbf{E}})$ is an     exact sequence whenever $U$  is an FP-injective left $R$-module. Then, in \cite{Z.G},  Gao and Wang introduced  other concept of Gorenstein FP-injective modules as follows: $M$ is said to be {\it Gorenstein FP-injective} \cite{Z.G} if there is an exact sequence $$
{\mathbf{E}}=\cdots\longrightarrow E_1 \longrightarrow E_{0} \longrightarrow E^0\longrightarrow E^1 \longrightarrow\cdots$$ of FP-injective left modules with $M=\ker(E^0\rightarrow E^1)$ such that ${\Hom}_{R}(P,{\mathbf{E}})$ is an  exact  sequence   whenever $P$  is a finitely presented module with ${\rm pd}_R(P)<\infty.$

  Let $G$ be a multiplicative group with neutral element $e$.
A graded ring $R$ is a ring with identity $1$ together with a direct decomposition
$R =\bigoplus_{\sigma\in G} R_{\sigma}$ (as additive subgroups) such that $R_{\sigma}R_{\tau}\subseteq R_{\sigma\tau}$ for all $\sigma , \tau\in G$.
Thus, $Re$ is a subring of $R$, $1\in Re$  and $R_{\sigma}$ is an $Re$-bimodule for every $\sigma\in G$.
A {\it graded} left (resp. right) $R$-module is a left (resp. right) $R$-module $M$ endowed with an internal direct sum
decomposition $M =\bigoplus_{\sigma\in G} M_{\sigma}$, where each $M_{\sigma}$ is a subgroup of the additive group
of $M$ such that $R_{\sigma}M_{\tau}\subseteq M_{\sigma\tau}$ for all $\sigma , \tau\in G$. For any graded left $R$-modules $M$
and $N$, set ${\rm Hom}_{R-{\rm gr}}(M,N) := \{f : M\rightarrow N  \mid$ $f \  is \  R$-linear  and   $f(M_{\sigma})\subseteq N_{\sigma}$ for   any $\sigma\in G\},$
which is the group of all morphisms from $M$ to $N$ in the class $R$-gr of all graded
left $R$-modules (gr-$R$ will denote the class of all graded right $R$-modules). It
is well known that $R$-gr is a Grothendieck category. An $R$-linear map $f : M\rightarrow N$
is said to be a {\it graded morphism of degree} $\tau$ with $\tau\in G$ if $f(M_{\sigma})\subseteq N_{\sigma\tau}$ for all
 $\sigma\in G$. Graded morphisms of degree $\sigma$ build an additive subgroup ${\rm HOM}_{R}(M,N)_{\sigma}$ of ${\Hom}_{R}(M,N)$. Then ${\rm HOM}_{R}(M,N) = \bigoplus_{\sigma\in G}{\rm HOM}_{R}(M,N)_{\sigma}$ is a graded abelian group of type $G$. We will denote by $\Ext_{R-{\rm gr}}^{i}$
 and ${\rm EXT}_{R}^i$ the right derived functors of ${\rm Hom}_{R-{\rm gr}}$ and ${\rm HOM}_{R}$, respectively.
Given a graded left $R$-module $M$, the {\it graded character module} of $M$ is defined as $M^{*} := {\rm HOM}_{\mathbb{Z}}(M,\mathbb{Q}/\mathbb{Z})$, where $\mathbb{Q}$ is the rational numbers
field and $\mathbb{Z}$ is the integers ring. It is easy to see that $M^{*} =\bigoplus_{\sigma\in G}{\rm HOM}_{\mathbb{Z}}(M_{\sigma^{-1}}, \mathbb{Q}/\mathbb{Z})$.

 Let $M$ be a graded right $R$-module and $N$ a graded left $R$-module. The abelian
group $M\otimes_{R}N$  may be graded by putting $(M
\otimes_{R} N)_{\sigma} $ with $\sigma\in G$ to be the additive
subgroup generated by elements $x\otimes y$
  with $x\in M_{\alpha}$ and $y\in N_{\beta}$ such that $\alpha\beta=\sigma$.
The object of $\mathbb{Z}$-gr thus defined will be called  the {\it graded tensor product} of $M$ and $N$.

If $M$ is a graded left $R$-module and $\sigma\in G$, then $M(\sigma)$ is the graded left $R$-module
obtained by putting $M(\sigma)_{\tau} = M_{\tau\sigma}$ for any $\tau\in G$. The graded module $M(\sigma)$ is
called the $\sigma$-{\it suspension} of $M$. We may regard the $\sigma$-suspension as an isomorphism
of categories $T_{\sigma}: R$-gr$ \  \rightarrow R$-gr, given on objects as $T_{\sigma}(M)=M(\sigma)$ for any $M\in$ $ R$-gr.
The forgetful functor ${U} : R$-gr$ \ \rightarrow R$-Mod associates to $M$ the underlying
ungraded $R$-module. This functor has a right adjoint $F$ which associated
to $M\in R$-Mod the graded $R$-module ${F}(M) = \bigoplus_{\sigma\in G}(^{\sigma}M)$, where each $^{\sigma}M$
is a copy of $M$ written $\{^{\sigma}x : x \in M\}$ with $R$-module structure defined by
$r{*}^{\tau}x =^{\sigma\tau}(rx)$ for each $r\in R_{\sigma}$. If $f : M\rightarrow N$ is $R$-linear, then ${F}(f) :{F}(M)\rightarrow {F}(N)$ is a graded morphism given by ${F}(f)(^{\sigma}x)=^{\sigma}f(x)$.

The injective (resp. flat) objects of $R$-gr (resp. gr-$R$) will be called {\it gr-injective} (resp. {\it gr-flat}) modules, because $M$ is gr-injective (resp. gr-flat) if and only if it is a injective
(resp. flat) graded module. By gr-$\pd_{R}(M)$ and gr-$\fd_{R}(M)$ we will denote the gr-projective and
gr-flat dimension of a graded module M, respectively. A graded left (resp. right) module $M$  is said to be {\it Gorenstein gr-injective} (resp. {\it Gorenstein gr-flat})
\cite{MJO, MJL, MJM} if there
is an exact sequence
 $$\cdots\longrightarrow I_{1}  \longrightarrow I_{0} \longrightarrow I^0\longrightarrow I^1\longrightarrow \cdots$$ of gr-injective (resp. gr-flat ) left (resp. right) modules
with $M= {\rm ker}(I^0\rightarrow
I^1)$ such that $\Hom_{R-{\rm gr}}(E,-)$ (resp. $-\otimes_{R}E$)  leaves the sequence
exact whenever $E$ is a gr-injective $R$-module.
The gr-injective envelope of $M$ is denoted by $E^g(M)$.  A graded left module $M$  is said to be {\it Ding gr-injective} \cite{L.X2} if there is an exact sequence $\cdots\rightarrow I_1\rightarrow I_{0}\rightarrow I^0\rightarrow I^1\rightarrow \cdots$ of gr-injective left modules, with $M= {\ker}(I^0\rightarrow
I^1)$ such that ${\rm Hom}_{R-{\rm gr}}(E,-) $  leaves the sequence exact whenever $E$ is an FP-gr-injective left $R$-module.

\begin{Def}[\cite{NG}, Definition 3.1] \label{NG}
Let $n \geq 0$ be an integer. Then, a graded left module $U$ is called
$n$-presented, if there exists an exact sequence
$
F_n\rightarrow F_{n-1} \rightarrow\cdots \rightarrow F_1\rightarrow F_0\rightarrow U\rightarrow 0$
in $R$-gr with each $F_i$ is finitely generated free left $R$-module.

Set $K_{n-1}={\rm Im}(F_{n-1}\rightarrow F_{n-2})$ and $K_{n}={\rm Im}(F_{n}\rightarrow F_{n-1})$. Then we get a short exact sequence $
0\rightarrow K_n  \rightarrow F_{n-1} \rightarrow K_{n-1}\rightarrow 0$ in $R$-gr with $F_{n-1}$ is a finitely generated free module. The modules $K_n$ and $K_{n-1}$ will be called {\it special finitely gr-generated} and special finitely gr-presented, respectively.  The sequence $(\Delta):
0\rightarrow K_n  \rightarrow F_{n-1} \rightarrow K_{n-1}\rightarrow 0$ in $R$-gr will be called a special short exact sequence.

Moreover, a short exact sequence $
0\rightarrow A \rightarrow B \rightarrow C\rightarrow 0$ in $R$-gr is called special gr-pure  if the induced sequence
$$
0\longrightarrow {\rm HOM}_{R}(K_{n-1},A) \longrightarrow  {\rm HOM}_{R}(K_{n-1},B)\longrightarrow {\rm HOM}_{R}(K_{n-1},C) \longrightarrow 0$$ is exact for every special finitely gr-presented module $K_{n-1}$. In this case  $A$ is said to be special gr-pure in $B$.

Analogously to the classical case, a graded ring $R$ is called left $n$-gr-coherent  if each $n$-presented module in $R$-gr is $(n + 1)$-presented.
\end{Def}

Ungraded $n$-presented modules have been used by many authors in order to extend some  homological notions. For 
example,  in \cite{BPCot}, let $R$ be an associative ring and $M$ a left $R$-module, then  module $M$ is called {\it $FP_{n}$-injective}  if ${\rm Ext}_{R}^{1}(L,M)=0$ for all $n$-presented left
$R$-modules $L$. In 2018,   Zhao,  Gao and   Huang in \cite{NG} showed that if we similarly use the derived functor ${\rm EXT}^{1}$ to define the $FP_n$-gr-injective
and $FP_{\infty}$-gr-injective modules, then they are just the $FP_n$-injective and
$FP_{\infty}$-injective objects in the class of graded modules, respectively. If $L$ is an $n$-presented graded left $R$-module with $n\geq 2$, then ${\rm EXT}_{R}^{1}(L,M)={\rm Ext}_{R}^{1}(L,M)$ for any graded $R$-module $M$. For this reason, they introduced the  concept of $n$-FP-gr-injective modules as follows: A graded left $R$-module
$M$ is called {\it $n$-FP-gr-injective} \cite{NG} if ${\rm EXT}_{R}^{n}(N,M)=0$ for any finitely $n$-presented graded
left $R$-module $N$. If $n=1$, then $M$ is {\it FP-gr-injective}. A graded right $R$-module
$M$ is called {\it $n$-gr-flat} \cite{NG}  if ${\rm Tor}_{R}^{n}(M,N)=0$ for any finitely $n$-presented graded left
$R$-module $N$.

If $U$ is an $n$-presented graded left $R$-module and $0\rightarrow K_n  \rightarrow F_{n-1} \rightarrow K_{n-1}\rightarrow 0$
is a special short exact sequence in $R$-gr with respect to $U$, then ${\rm EXT}_{R}^{n}(U,M)\cong{\rm EXT}_{R}^{1}(K_{n-1},M)$ for every graded left $R$-module $M$, and ${\rm Tor}_{n}^{R}(M,U)\cong{\rm Tor}_{n}^{R}(M, K_{n-1})$ for every graded right $R$-module $M$. The $n$-FP-gr-injective dimension of a graded
left $R$-module $M$, denoted by $n.{\rm FP}$-${\rm gr}$-${\rm id}_{R}(M)$, is defined to be the least integer $k$ such
that ${\rm EXT}_{R}^{k+1}(K_{n-1}, M)=0$  for any special gr-presented module
$K_{n-1}$ in $R$-gr. The $n$-${\rm gr}$-flat dimension of a graded
right $R$-module $M$, denoted by $n$-gr-${\rm fd}_{R}(M)$, is defined to be the least integer $k$ such
that ${\rm Tor}_{k+1}^{R}(M, K_{n-1})=0$  for any special gr-presented module $K_{n-1}$ in $R$-gr. Also, $r.n.{\rm FP}$-${\rm gr}$-${\rm dim}(R)$=sup\{$n.{\rm FP}$-${\rm gr}$-${\id_{R}}(M)\ \mid {\rm \  M \ is \ a \ graded \ left \  module}\}$ and
$r.n$-${\rm gr}$-${\rm dim}(R)$=sup\{ $n$-${\rm gr}$-${\fd_{R}}(M) \ \mid {\rm \  M \ is \ a \ graded \ right \ module}\}$.
\section{Gorenstein $n$-FP-gr-Injective and Gorenstein $n$-gr-Flat Modules}
\ \ In this section, we introduce and study   Gorenstein $n$-FP-gr-injective and  Gorenstein $n$-gr-flat modules which are defined as follows:
\begin{Def}\label{2.2}
Let $R$ be a graded ring and $n\geq 1$ an integer. Then,
  a  module $M$ in $R$-gr  is  called {\it  Gorenstein $n$-FP-gr-injective} if there exists an    exact sequence of  $n$-FP-gr-injective modules in $R$-gr of this form:
$${\mathbf{A}}= \cdots\longrightarrow A_1\longrightarrow A_{0}\longrightarrow A^0\longrightarrow
A^1\longrightarrow\cdots$$ with $M=\ker(A^0\rightarrow A^1)$ such that ${\rm HOM}_{R}(K_{n-1},{\mathbf{A}})$  is an exact sequence whenever $K_{n-1}$  is a special gr-presented module in $R$-gr with gr-${\rm pd}_R(K_{n-1})<\infty.$

The class of Gorenstein $n$-FP-gr-injective will be denoted $\mathcal{G}_{gr-\mathcal{FI}_n}$.

  A module $N$ in gr-$R$  is called  Gorenstein $n$-gr-flat  if there exists the following exact sequence of $n$-gr-flat  modules in gr-$R$ of this form:
$${\mathbf{F}}= \cdots\longrightarrow F_1\longrightarrow F_{0}\longrightarrow F^0\longrightarrow
F^1\longrightarrow\cdots$$ with $N=\ker(F^0\rightarrow F^1)$ such that ${\mathbf{F}}\otimes_{R}K_{n-1}$   is an exact sequence   whenever $K_{n-1}$  is a special gr-presented module in $R$-gr with gr-${\rm fd}_R(K_{n-1})<\infty.$

The class of Gorenstein $n$-FP-gr-flat will be denoted $\mathcal{G}_{gr-\mathcal{F}_n}$.

\end{Def}

In the    ungraded case,  the $R$-modules $A_i$ and $A^i$ (resp. $F_i$ and $F^i$) as in  the definition above,  are called  $n$-FP-injective (resp. $n$-flat).  Also, $R$-modules $M$ and $N$ are called  Gorenstein $n$-FP-injective and Gorenstein $n$-flat, respectively, and $K_{n-1}$ is a special presented left module with respect to any $n$-presented left $R$-module $U$.

\begin{rem}\label{21}
Let $R$ be a graded ring. Then:
\begin{enumerate}
\item [\rm (1)] $gr$-$\mathcal{I}\subseteq gr$-$\mathcal{FI}_{1}\subseteq gr$-$\mathcal{FI}_{2}\subseteq\cdots\subseteq gr$-$\mathcal{FI}_{n}\subseteq\mathcal{G}_{gr-\mathcal{FI}_n}$. But, Gorenstein $n$-FP-gr-injective $R$-modules need not be
gr-injective, see Example \ref{ex2}(1). Also,
 $gr$-$\mathcal{F}\subseteq gr$-$\mathcal{F}_{1}\subseteq gr$-$\mathcal{F}_{2}\subseteq\cdots\subseteq gr$-$\mathcal{F}_{n}\subseteq\mathcal{G}_{gr-\mathcal{F}_n}$.

In general, every Gorenstein $n$-FP-gr-injective (resp. Gorenstein $n$-gr-flat) $R$-module  is not $n$-FP-gr-injective (resp. $n$-gr-flat),   except in a certain state, see Proposition \ref{2.2a}.
\item [\rm (2)]  $\mathcal{G}_{gr-\mathcal{FI}_1}\subseteq\mathcal{G}_{gr-\mathcal{FI}_2}\subseteq\cdots\subseteq\mathcal{G}_{gr-\mathcal{FI}_n}$ and $\mathcal{G}_{gr-\mathcal{F}_1}\subseteq\mathcal{G}_{gr-\mathcal{F}_2}\subseteq\cdots\subseteq\mathcal{G}_{gr-\mathcal{F}_n}$.  But for any integers $m> n$, Gorenstein $m$-FP-gr-injective (resp. Gorenstein $m$-gr-flat) $R$-modules need not be
Gorenstein $n$-FP-gr-injective (resp. Gorenstein $n$-gr-flat), see Example \ref{ex2}(2, 3).
\item [\rm (3)]
In Definition \ref{2.2}, it is clear that
 $\ker(A_i\rightarrow A_{i-1})$ and $\ker(A^i\rightarrow A^{i+1})$ are Gorenstein $n$-FP-gr-injective, and $\ker(F_i\rightarrow F_{i-1})$, $\ker(F^i\rightarrow F^{i+1})$ are Gorenstein $n$-gr-flat for any $i\geq 1$.
\end{enumerate}
\end{rem}
It is known that the trivial extension of a commutative ring $A$ by an $A$-module $M$,  $R=A\ltimes M$, is a $\mathbb{Z}_{2}$-graded ring, see \cite{DAB, DAA}.

\begin{ex}\label{ex2}
(1)
 Let $K$ be a field with characteristic $p\neq 0$ and let $G = \cup_{k\geq1}G_k$, where
$G_k$ is the cyclic group with generator $a_k$, the order of $a_k$ is $p^k$ and $a_k = a_{k+1}^p$. Let $R = K[G]$.  Then, by Remark \ref{21},  $R[H]$ is Gorenstein $n$-FP-gr-injective for every group $H$, since by \cite[Example iii]{MJP}, $R[H]$ is $n$-FP-gr-injective but it is not gr-injective.

(2) Let $A$ be a field, $E$ a nonzero $A$-vector space  and $R=A\ltimes E$ be a trivial extension of $A$ by $E$.  If $\dim _{A}E=1$, then by Remark \ref{21}, every $R$-module in $R$-gr is Gorenstein $n$-FP-gr-injective, see \cite[Corollary 2.2]{KMK}. If $E$ is  an $A$-vector space with infinite rank, then by \cite[Theorem 3.4]{L.g}, every $2$-presented module in $R$-gr is projective. So, every module in $R$-gr is $2$-FP-gr-injective and hence, every module in $R$-gr is Gorenstein $2$-FP-gr-injective. If every module in $R$-gr is Gorenstein $1$-FP-gr-injective, then $R$ is gr-regular,  contradiction.

(3)  Let $R=k[X]$, where $k$ is a field. Then, by Theorem \ref{thm2}, every graded right $R$-module is Gorenstein $2$-gr-flat, and there is a graded right $R$-module that is not Gorenstein $1$-gr-flat, 
since $l.$FP- gr-$ dim(R)\leq1$, see Proposition \ref{2.2a} and \cite[Example 3.6]{NG}.
\end{ex}

We start with the result which proves that the behaviour of Gorenstein $n$-FP-gr-injective (resp. Gorenstein $n$-gr-flat) modules in short exact sequences is the same as  the one of the classical homological notions.

\begin{prop}\label{prop1}
Let $R$ be a graded ring. Then:
\begin{enumerate}
\item [\rm (1)]
For every short exact sequence   $0 \rightarrow A\rightarrow B\rightarrow C\rightarrow 0$ in $R$-gr, $B$ is Gorenstein $n$-FP-gr-injective if $A$ and $C$ are  Gorenstein $n$-FP-gr-injective.
\item [\rm (2)]
For every short exact sequence   $0 \rightarrow A\rightarrow B\rightarrow C\rightarrow 0$ in gr-$R$, $B$ is Gorenstein $n$-gr-flat if $A$ and $C$ are  Gorenstein $n$-gr-flat.
\end{enumerate}
\end{prop}
\begin{proof}
 (1) By Definition \ref{2.2}, there is an  exact sequence $\cdots\rightarrow A_1\rightarrow A_{0}\rightarrow A^{0}\rightarrow A^{1}\rightarrow\cdots$ of $n$-FP-gr-injective modules in $R$-gr, where $A={\Ker}(A^{0}\rightarrow A^{1})$, $K_{i}^{'}={\Ker}(A_{i}\rightarrow A_{i-1})$ and $(K^{i})^{'}={\Ker}(A^{i}\rightarrow A^{i+1})$. Also,  there is an  exact sequence $\cdots\rightarrow C_1\rightarrow C_{0}\rightarrow C^{0}\rightarrow C^{1}\rightarrow\cdots$ of $n$-FP-gr-injective modules in $R$-gr, where $C={\Ker}(C^{0}\rightarrow C^{1})$, $K_{i}^{''}={\Ker}(C_{i}\rightarrow C_{i-1})$ and $(K^{i})^{''}={\Ker}(C^{i}\rightarrow C^{i+1})$.
For any $n$-presented graded left module $P$, ${\rm EXT}_{R}^{n}(P,A_{i}\oplus C_{i})={\rm EXT}_{R}^{n}(P,A_{i})\oplus{\rm EXT}_{R}^{n}(P, C_{i}) =0$, then $A_{i}\oplus C_{i}$ is $n$-FP-gr-injective for any $i\geq 0$. Similarly,  $A^{i}\oplus C^{i}$ is $n$-FP-gr-injective for any $i\geq 0$. Therefore, there is an exact sequence
$${\mathbf{\Y}}= \cdots\longrightarrow A_1\oplus C_1\longrightarrow A_{0}\oplus C_{0}\longrightarrow A^{0}\oplus C^{0}\longrightarrow
A^{1}\oplus C^{1}\longrightarrow\cdots$$
of $n$-FP-gr-injective modules in $R$-gr, where $B={\Ker}(A^{0}\oplus C^{0}\rightarrow A^{1}\oplus C^{1})$, $K_{i}=K_{i}^{'}\oplus K_{i}^{''}={\Ker}(A_{i}\oplus C_{i}\rightarrow A_{i-1}\oplus C_{i-1})$ and $K^{i}=(K^{i})^{'}\oplus (K^{i})^{''}={\Ker}(A^{i}\oplus C^{i}\rightarrow A^{i+1}\oplus C^{i+1})$. Let $K_{n-1}$ be a special gr-presented module in $R$-gr with gr-${\pd}_R(K_{n-1})<\infty$. Then, ${\rm EXT}_{R}^{1}(K_{n-1}, B)=0$, and also we have:
${\rm EXT}_{R}^{1}(K_{n-1}, K_{i})={\rm EXT}_{R}^{1}(K_{n-1}, K_{i}^{'}\oplus K_{i}^{''})=0$.
Similarly, ${\rm EXT}_{R}^{1}(K_{n-1}, K^{i})=0$. Consequently, ${\rm HOM}_{R}(K_{n-1}, {\mathbf{\Y}})$ is  exact and so, $B$ is Gorenstein $n$-FP-gr-injective.

(2) By Definition \ref{2.2}, there is an  exact sequence $\cdots\rightarrow A_1\rightarrow A_{0}\rightarrow A^{0}\rightarrow A^{1}\rightarrow\cdots$ of $n$-gr-flat modules in gr-$R$, where $A={\Ker}(A^{0}\rightarrow A^{1})$, $K_{i}^{'}={\Ker}(A_{i}\rightarrow A_{i-1})$ and $(K^{i})^{'}={\Ker}(A^{i}\rightarrow A^{i+1})$. Also,  there is an  exact sequence $\cdots\rightarrow C_1\rightarrow C_{0}\rightarrow C^{0}\rightarrow C^{1}\rightarrow\cdots$ of $n$-gr-flat modules in gr-$R$, where $C={\Ker}(C^{0}\rightarrow C^{1})$, $K_{i}^{''}={\Ker}(C_{i}\rightarrow C_{i-1})$ and $(K^{i})^{''}={\Ker}(C^{i}\rightarrow C^{i+1})$.
Similarly to (1), there is an exact sequence
$${\mathbf{\Y}}= \cdots\longrightarrow A_1\oplus C_1\longrightarrow A_{0}\oplus C_{0}\longrightarrow A^{0}\oplus C^{0}\longrightarrow
A^{1}\oplus C^{1}\longrightarrow\cdots$$
of $n$-gr-flat modules in gr-$R$, where $B={\Ker}(A^{0}\oplus C^{0}\rightarrow A^{1}\oplus C^{1})$, and if $K_{n-1}$ is a special gr-presented module in $R$-gr with gr-${\fd}_R(K_{n-1})<\infty$, then  ${\mathbf{\Y}}\otimes_{R}K_{n-1}$ is  exact and so, $B$ is Gorenstein $n$-gr-flat.
\end{proof}

Transfer results  of $n$-FP-injective and Gorenstein $n$-FP-injective modules with respect to the functor $F$ is given in the following result.

\begin{prop}\label{prop2}
Let $R$ be a ring graded by a group $G$.
\begin{enumerate}
\item [\rm (1)]
If $M$  is an $n$-FP-injective left $R$-module,
 then $F(M)$ is $n$-FP-gr-injective.
\item [\rm (2)]
If $M$  is a Gorenstein $n$-FP-injective left $R$-module,
 then $F(M)$ is Gorenstein $n$-FP-gr-injective.
 \end{enumerate}
\end{prop}
\begin{proof}
(1) If $0\rightarrow K_{n}\rightarrow F_{n-1}\rightarrow K_{n-1}\rightarrow 0$ is special short exact sequence in $R$-gr with respect to an $n$-presented graded left $R$-module $U$, then similar to the proof of \cite[Lemma 2.3]{ZL}, $0={\rm Ext}_{R-{\rm gr}}^{1}(K_{n-1}, F(M)(\sigma))={\rm Ext}_{R-{\rm gr}}^{n}(U, F(M)(\sigma))$, and hence by \cite[Proposition 3.10]{NG}, $F(M)$ is $n$-FP-gr-injective.

(2) Let $M$ be a Gorenstein $n$-FP-injective left $R$-module. Then, there exists an exact sequence of  $n$-FP-injective left modules :
$${\mathbf{B}}= \cdots\longrightarrow B_1\longrightarrow B_{0}\longrightarrow B^0\longrightarrow
B^1\longrightarrow\cdots$$ with $M=\ker(B^0\rightarrow B^1)$ such that ${\rm Hom}_{R}(K_{n-1}^{'},{\mathbf{B}})$  is an exact sequence   whenever $K_{n-1}^{'}$  is  a special finitely presented module in $R$-gr with ${\rm pd}_R(K_{n-1}^{'})<\infty.$ By (1), $F(B_i)$ and $F(B^i)$ are $n$-FP-gr-injective for any $i\geq 0$. Since the functor $F$ is exact, we get the following exact sequence
$${\mathbf{F(B)}}= \cdots\longrightarrow F(B_1)\longrightarrow F(B_{0})\longrightarrow F(B^0)\longrightarrow
F(B^1)\longrightarrow\cdots$$ of $n$-FP-gr-injective left $R$-modules with $F(M)=\ker(F(B^0)\rightarrow F(B^1))$. If $K_{n-1}$ is  special gr-presented left module with gr-${\pd}_{R}(K_{n-1})<\infty$, then $U(K_{n-1})$ is finitely presented  with ${\pd}_{R}(U(K_{n-1}))<\infty$.  By hypothesis, ${\Hom}_{R}(U(K_{n-1}), {\mathbf{B}})$ is exact. Therefore, from ${\Hom}_{R}(U(K_{n-1}), {\mathbf{B}})={\Hom}_{R-{\rm gr}}(K_{n-1}, {\mathbf{F(B)}})$, it follows that ${\Hom}_{R-{\rm gr}}(K_{n-1}, {\mathbf{F(B)}})$ is exact and consequently, the isomorphism $${\rm HOM}_{R}(K_{n-1}, {\mathbf{F(B)}})=\bigoplus_{\sigma\in G}{\rm HOM}_{R}(K_{n-1}, {\mathbf{F(B)}})_{\sigma}\cong \bigoplus_{\sigma\in G}{\Hom}_{R-{\rm gr}}(K_{n-1}, {\mathbf{{\mathbf{F(B)}}(\sigma)}})$$ implies  that $F(M)$ is Gorenstein $n$-FP-gr-injective.
\end{proof}

Now, we give a characterization of a graded ring $R$ on which $n$-presented modules in $R$-gr with gr-${\rm pd}_{R}(U)<\infty$ (resp. gr-${\rm fd}_{R}(U)<\infty$) are $(n+1)$-presented. For this, we need the following lemma.

\begin{lem}\label{lem1}
Assume that  every $n$-presented module in $R$-gr with gr-${\rm fd}_{R}(U)<\infty$ is $(n+1)$-presented. Then, for any $t\geq 1$:
\begin{enumerate}
\item [\rm (1)]
${\rm EXT}_{R}^{t}(K_{n-1}, M)=0$ for any Gorenstein $n$-FP-gr-injective left $R$-module $M$ and  any
special gr-presented left $R$-module $K_{n-1}$ with gr-${\rm pd}_{R}(K_{n-1})<\infty$.
\item [\rm (2)]
${\rm Tor}_{t}^{R}(M, K_{n-1})=0$ for any Gorenstein $n$-gr-flat right $R$-module $M$  and any
special gr-presented left $R$-module $K_{n-1}$ with gr-${\rm fd}_{R}(K_{n-1})<\infty$.
\end{enumerate}
\end{lem}
\begin{proof}
(1)    Assume that $ K_{n-1}$ is a special gr-presented module in $R$-gr with gr-${\rm pd}_{R}(K_{n-1})\leq m$ respect to any $n$-presented module $U$ in $R$-gr. If  $M$ is a  Gorenstein $n$-FP-gr-injective left $R$-module,
then, there is a left $n$-FP-gr-injective resolution of $M$ in $R$-gr. So, we have:
$$0\longrightarrow N\longrightarrow E_{m-1}\longrightarrow \cdots \longrightarrow E_{0}\longrightarrow M\longrightarrow 0,$$ where every $E_j$ is $n$-FP-gr-injective for every $0 \leq j \leq m-1$.   Since gr-${\rm fd}_{R}(U)<\infty$, $U$ is $(n+1)$-presented, and so 
 ${\rm EXT}_{R}^{i+1}(K_{n-1}, E_j)= 0$ for any $i\geq 0$. Hence,
 ${\rm EXT}_{R}^{i+1}(K_{n-1},M)\cong{\rm EXT}_{R}^{m+i+1}(K_{n-1}, N)$, and since gr-${\rm pd}_{R}(K_{n-1})\leq m$, it follows that , ${\rm EXT}_{R}^{i+1}(K_{n-1},M)=0$ for any $i\geq 0$.

 (2) Assume that $ K_{n-1}$ is a special gr-presented module  in $R$-gr with gr-${\rm fd}_{R}(K_{n-1})\leq m$ for any $n$-presented module $U$ in $R$-gr.   If  $M$ is a  Gorenstein $n$-gr-flat right $R$-module,
then there is a right $n$-gr-flat resolution of $M$ in gr-$R$ of the form:
$$0\longrightarrow M\longrightarrow F^{0}\longrightarrow \cdots \longrightarrow F^{m-1}\longrightarrow N\longrightarrow 0,$$ where every $F^j$ is $n$-gr-flat for every $0 \leq j \leq m-1$.  Since $U$ is  $(n+1)$-presented, we have
 ${\rm Tor}_{i+1}^{R}(F^{j}, K_{n-1})= 0$ for any $i\geq 0$. If gr-${\rm fd}_{R}(K_{n-1})\leq m$, then
 ${\rm Tor}_{i+1}^{R}(M, K_{n-1})\cong{\rm Tor}_{m+i+1}^{R}(N, K_{n-1})=0$, and so ${\rm Tor}_{i+1}^{R}(M, K_{n-1})=0 $  for any $i\geq 0$.
 \end{proof}
 
\begin{thm}\label{thm1}
Let  $R$ be a graded ring. Then,  the following statements are equivalent:
\begin{enumerate}
\item [\rm (1)]
Every $n$-presented module in $R$-gr with gr-${\rm pd}_{R}(U)<\infty$ is $(n+1)$-presented;
\item [\rm (2)]
For every short exact sequence   $0 \rightarrow A\rightarrow B\rightarrow C\rightarrow 0$ in $R$-gr, $C$ is Gorenstein $n$-FP-gr-injective if $A$ and $B$ are  Gorenstein $n$-FP-gr-injective.
\end{enumerate}
\end{thm}
\begin{proof}
$(1)\Longrightarrow (2)$
If $B$ is a Gorenstein $n$-FP-gr-injective module in $R$-gr, then by Definition \ref{2.2} and Remark \ref{21}, there is an exact sequence $0 \rightarrow K\rightarrow B_0\rightarrow B\rightarrow 0$ in $R$-gr, where $B_0$ is $n$-FP-gr-injective and $K$ is Gorenstein $n$-FP-gr-injective. Consider the following commutative diagram with exact rows exists:
$$\xymatrix{
 &  0 \ar[d] & 0\ar[d]& &  \\
0 \ar[r]& K \ar@{=}[r]\ar[d]&K\ar[d]&& \\
0 \ar[r] & D\ar[r]\ar[d]& B_{0}\ar[r]\ar[d]&C\ar[r]\ar@{=}[d]& 0 \\
0 \ar[r]& A\ar[r]\ar[d]& B\ar[r]\ar[d]&C \ar[r]\ar[d]& 0 \\
 & 0  & 0  &0 & \\
}$$
By Proposition \ref{prop1}(1), $D$ is Gorenstein $n$-FP-gr-injective, and so we have a  commutative diagram in  $R$-gr:
\begin{displaymath}\xymatrix{
   \cdots\ar[r] & D_{1} \ar[rr]\ar[dr]&& D_{0} \ar[rr]\ar[dr]&& B_{0} \ar[rr]\ar[rd]&& E^{0}\ar[rr]\ar[rd]&& E^{1}\ar[r]&\cdots,\\
     & &  L_{0}\ar[ur] \ar[rd] && D\ar[ur]\ar[rd] & &C \ar[ur] \ar[rd] & &L^{1} \ar[ur] \ar[rd]\\
      & 0\ar[ur] && 0 \ar[ur]        && 0\ar[ur]          && 0\ar[ur]   &&  0}
\end{displaymath}
where $D_i$ and $B_{0}$ are $n$-FP-gr-injective, $E^{i}$ is gr-injective, $C={\rm Ker}(E^{0}\rightarrow E^{1})$, $D={\rm Ker}(B_{0}\rightarrow C)$, $L_{i}={\rm Ker}(D_{i}\rightarrow D_{i-1})$ and $L^{i}={\rm Ker}(E^{i}\rightarrow E^{i+1})$.
By Remark \ref{21}, $E^{i}$ and $L_{i}$ are Gorenstein $n$-FP-gr-injective and hence by Lemma \ref{lem1}(1),  ${\rm EXT}_{R}^{t}(K_{n-1}, L_{i})={\rm EXT}_{R}^{t}(K_{n-1}, D)=0$ for any special gr-presented $K_{n-1}$ module in $R$-gr with gr-${\rm pd}_{R}(K_{n-1})<\infty$. Therefore,
we have the following exact commutative diagram :
\begin{displaymath}\xymatrix{
  \Hom(K_{n-1},D_{1}) \ar[rr]\ar[dr]&& \Hom(K_{n-1},D_{0}) \ar[rr]\ar[rd]&& \cdots \\
   & \Hom(K_{n-1},L_{0}) \ar[ur] \ar[rd] & &\Hom(K_{n-1},D) \ar[ur] \ar[rd] \\
   0\ar[ur] &&        0\ar[ur]  && 0}
\end{displaymath}
Hence, $C$ is Gorenstein $n$-FP-gr-injective.

$(2)\Longrightarrow (1)$ Let $U$ be an $n$-presented graded left $R$-module with gr-${\rm pd}_{R}(U)<\infty$, and let $0 \rightarrow K_{n}\rightarrow F_{n-1}\rightarrow K_{n-1}\rightarrow 0$ be a special short exact sequence in $R$-gr with respect to $U$, where $K_{n}$ is a special gr-generated module. We show that $K_n$ is special gr-presented.
 Let $M$ be a Gorenstein $n$-FP-injective module and $0 \rightarrow M\rightarrow E\rightarrow L\rightarrow 0$ an exact sequence in $R$-Mod, where $E$ is injective. Then, $0 \rightarrow F(M)\rightarrow F(E)\rightarrow F(L)\rightarrow 0$ is exact, where $F(M)$ and $F(E)$ are Gorenstein $n$-FP-gr-injective in $R$-gr by Proposition \ref{prop2}. So by (2), we deduce that $F(L)$ is Gorenstein $n$-FP-gr-injective.
We have:
$$0={\rm Ext}_{R-{\rm gr}}^{1}(F_{n-1}, F(M))\longrightarrow{\Ext}_{R-{\rm gr}}^{1}(K_{n},F(M))\longrightarrow{\rm Ext}_{R-{\rm gr}}^{2}(K_{n-1},F(M))\longrightarrow 0.$$ So, ${\Ext}_{R-{\rm gr}}^{1}( K_{n},F(M))\cong{\rm Ext}_{R-{\rm gr}}^{2}(K_{n-1},F(M)).$
On the other hand,
$$0={\rm Ext}_{R-{\rm gr}}^{1}(K_{n-1},F(E))\longrightarrow{\rm Ext}_{R-{\rm gr}}^{1}(K_{n-1},F(L))\longrightarrow{\rm Ext}_{R-{\rm gr}}^{2}(K_{n-1},F(M))\longrightarrow0.$$ Hence, ${\rm Ext}_{R-{\rm gr}}^{1}(K_{n-1},F(L))\cong{\rm Ext}_{R-{\rm gr}}^{2}(K_{n-1},F(M))$.  Since $F(L)$ is Gorenstein $n$-FP-gr injective,  we get
$0={\rm EXT}_{R}^{1}(K_{n-1},F(L))_{\sigma}\cong{\rm Ext}_{R-{\rm gr}}^{1}(K_{n-1},F(L)({\sigma}))$ for any $\sigma\in G$. This implies that ${\rm Ext}_{R-{\rm gr}}^{1}(K_{n-1},F(L))=0$ and consequently ${\Ext}_{R-{\rm gr}}^{1}( K_{n},F(M))=0$.
So, the following commutative diagram exists:
$$\xymatrix{
 0\ar[r] & {\rm Hom}_{R-{\rm gr}}(K_n,F(M)) \ar[r]\ar[d]^{\cong}& {\rm Hom}_{R-{\rm gr}}(K_n,F(E))\ar[r]\ar[d]^{\cong}&{\rm Hom}_{R-{\rm gr}}(K_n,F(L)) \ar[r]\ar[d]^{\cong}& 0 \\
0 \ar[r]& {\rm Hom}_{R}(K_n,M)\ar[r]&  {\rm Hom}_{R}(K_n,E) \ar[r]&{\rm Hom}_{R}(K_n,L)
 \\
}$$
So, ${\Ext}_{R-{\rm gr}}^{1}(K_n,F(M))\cong{\Ext}_{R}^{1}(K_n,M)=0$ for any Gorenstein $n$-FP-injective left $R$-module $M$. Since every FP-injective left module is Gorenstein $n$-FP-injective,  ${\Ext}_{R-{\rm gr}}^{1}(K_n,F(N))\cong{\Ext}_{R}^{1}(K_n,N)=0$ for any FP-injective left module $N$ and so $K_n$ is $1$-presented. Therefore, $U$ is $(n+1)$-presented in $R$-gr.
\end{proof}

\begin{cor}\label{cor1}
Let every $n$-presented module in $R$-gr with gr-${\rm pd}_{R}(U)<\infty$ be $(n+1)$-presented. Then, 
 a module $M$ in $R$-gr is Gorenstein $n$-FP-gr-injective if and only if every gr-pure submodule and any gr-pure epimorphic image of $M$ are Gorenstein $n$-FP-gr-injective.
\end{cor}
\begin{proof}
$(\Longrightarrow)$ Let $M$ be a Gorenstein $n$-FP-gr-injective module  in $R$-gr. If the exact sequence $0\rightarrow K\rightarrow M\rightarrow \frac{M}{K}\rightarrow 0$ is gr-pure, then by \cite[Proposition 2.2]{MJP}, ${\rm EXT}_{R}^1(K_{n-1}, K)=0$ for every special gr-presented module $K_{n-1}$ in $R$-gr. So, we have
 $0={\rm EXT}_{R}^1(K_{n-1}, K)\cong{\rm EXT}_{R}^n(U, K)$ for any $n$-presented module $U$ in $R$-gr. Thus, $K$ is $n$-FP-gr-injective, and hence $K$ is Gorenstein $n$-FP-gr-injective by Remark \ref{21}. Therefore, by Theorem \ref{thm1},  $\frac{M}{K}$ is  Gorenstein $n$-FP-gr-injective.

 $(\Longleftarrow)$ Assume that the exact sequence $0\rightarrow K\rightarrow M\rightarrow L\rightarrow 0$ in $R$-gr is gr-pure, where $L$ and $K$ are Gorenstein $n$-FP-gr-injective.
Then, by Proposition \ref{prop1}(1),  $M$ is  Gorenstein $n$-FP-gr-injective.
\end{proof}

The following definition is the graded version of \cite{EOO, L.DD}.

\begin{Def}\label{2.pa}
Let $\eth$ be a class of graded left $R$-module. Then:
\begin{enumerate}
\item [\rm (1)]
${\eth}^{\bot}={\rm Ker Ext}_{R-{\rm gr}}^{1}(\eth, -)=\{C \mid {\rm Ext}_{R-{\rm gr}}^{1}(L, C)=0 \ \ for \ \ any \ \ L\in\eth\}$.
\item [\rm (2)]
$^{\bot}{\eth}={\rm Ker Ext}_{R-{\rm gr}}^{1}(-, \eth)=\{C \mid {\rm Ext}_{R-{\rm gr}}^{1}(C, L)=0 \ \ for \ \ any \ \ L\in\eth\}$.

 A pair $(\F, \mathcal{C})$ of classes of graded $R$-modules is  called a cotorsion theory, if $\F^{\bot}= \mathcal{C}$ and $\F= {^{\bot}\mathcal{C}}$.
A cotorsion theory $(\F, \mathcal{C})$ is called hereditary,  if whenever $0 \rightarrow F^{'}\rightarrow F\rightarrow F^{''}\rightarrow 0$ is exact in $R$-gr with $F, F^{''}\in \F$ then $F^{'}$ is also in $\F$, or equivalently, if $0 \rightarrow C^{'}\rightarrow C\rightarrow C^{''}\rightarrow 0$ is an exact sequence in $R$-gr with $C, C^{'}\in \mathcal{C}$, then $C^{''}$ is also in $\mathcal{C}$.
\end{enumerate}
\end{Def}
\begin{cor}\label{2.vc}
Let  $R$ be a graded ring. Then,  the following statements are equivalent:
\begin{enumerate}
\item [\rm (1)]
$(^{\bot}\mathcal{G}_{gr-\mathcal{FI}_n}, \mathcal{G}_{gr-\mathcal{FI}_n})$ is a hereditary cotorsion pair;
\item [\rm (2)]
Every $n$-presented
module in $R$-gr with  gr-${\rm pd}(U)<\infty$ is $(n + 1)$-presented and every $M\in(^{\bot}\mathcal{G}_{gr-\mathcal{FI}_n})^{\bot}$ has an exact left ($gr$-$\mathcal{FI}_{n}$)-resolution.
\end{enumerate}
\end{cor}
\begin{proof}
$(1)\Longrightarrow(2)$
 Let $M$ be a Gorenstein $n$-FP-injective left $R$-module and $0 \rightarrow M\rightarrow E\rightarrow L\rightarrow 0$ an exact sequence in $R$-Mod, where $E$ is injective. Then, $0 \rightarrow F(M)\rightarrow F(E)\rightarrow F(L)\rightarrow 0$ is exact in $R$-gr, where $F(M)$ and $F(E)$ are Gorenstein $n$-FP-gr-injective by Proposition \ref{prop2}. So by hypothesis, $F(L)$ is  Gorenstein $n$-FP-gr-injective.  If $U$ is an $n$-presented graded left $R$-module with gr-${\rm pd}_{R}(U)<\infty$, then similar to the proof  $(2)\Longrightarrow(1)$ of Theorem \ref{thm1}, it follows that $U$ is $(n+1)$-presented.
 Since $(^{\bot}\mathcal{G}_{gr-\mathcal{FI}_n})^{\bot}=\mathcal{G}_{gr-\mathcal{FI}_n}$  and every $N\in\mathcal{G}_{gr-\mathcal{FI}_n}$ has an left exact  ($gr$-$\mathcal{FI}_{n}$)-resolution, then  $M\in(^{\bot}\mathcal{G}_{gr-\mathcal{FI}_n})^{\bot}$ as well.

$(2)\Longrightarrow(1)$ Note that we have to show that $(^{\bot}\mathcal{G}_{gr-\mathcal{FI}_n})^{\bot}=\mathcal{G}_{gr-\mathcal{FI}_n}$.  If
$M\in(^{\bot}\mathcal{G}_{gr-\mathcal{FI}_n})^{\bot}$, then
 $n$-FP-gr-injective resolution $ \cdots\rightarrow A_{3}\rightarrow A_{1}\rightarrow A_{0}\rightarrow M\rightarrow 0$  of $M$ in $R$-gr exists. Also, we have an  exact sequence $ 0\rightarrow M\rightarrow E_{0}\rightarrow E_{1}\rightarrow\cdots$ in $R$-gr, where any $E_{i}$ is gr-injective. So,  there exists an exact sequence
$$\Y:  \cdots\longrightarrow A_{3}\longrightarrow A_{1}\longrightarrow A_{0}\longrightarrow E_{0}\longrightarrow E_{1}\longrightarrow\cdots$$
of $n$-FP-gr-injective modules in $R$-gr with $M={\rm Ker}(E_{0}\rightarrow E_{1})$. Let $0\rightarrow K_{n}\rightarrow F_{n-1}\rightarrow K_{n-1}\rightarrow 0$ be a special short exact sequence in $R$-gr with gr-${\pd}_{R}(K_{n-1})<\infty$. Then, by hypothesis, $K_{n}$ is a  gr-presented module  with gr-${\pd}_{R}(K_{n})<\infty$. So, by \cite[Theorem 6.10]{Rot2} and by using the inductive presumption on gr-${\pd}_{R}(K_{n-1})$, we deduce that ${\rm HOM}_{R}(K_{n-1}, \Y)$ is exact. Thus, $M$ is Gorenstein $n$-FP-gr-injecive and hence $M\in\mathcal{G}_{gr-\mathcal{FI}_n}$.

Now, if $0\rightarrow A\rightarrow B\rightarrow C\rightarrow 0$ is a short exact sequence in
$R$-gr, where $A,B\in\mathcal{G}_{gr-\mathcal{FI}_n}$, then by Theorem \ref{thm1}, $C\in\mathcal{G}_{gr-\mathcal{FI}_n}$. Hence,  the pair $(^{\bot}\mathcal{G}_{gr-\mathcal{FI}_n}, \mathcal{G}_{gr-\mathcal{FI}_n})$ is a hereditary cotorsion pair.
\end{proof}

\begin{prop}\label{prop23}
Assume that every $n$-presented module in $R$-gr with gr-${\rm fd}_{R}(U)<\infty$ is $(n+1)$-presented. Then, 
for every short exact sequence $0 \rightarrow A\rightarrow B\rightarrow C\rightarrow 0$ in gr-$R$, $A$ is Gorenstein $n$-gr-flat if $B$ and $C$ are Gorenstein $n$-gr-flat.
\end{prop}
\begin{proof}
 If $B$ is a  Gorenstein  $n$-gr-flat module in gr-$R$, then by Definition \ref{2.2} and Remark \ref{21}, there is an exact sequence $0 \rightarrow B\rightarrow F^0\rightarrow L\rightarrow 0$ in gr-$R$, where $F^0$ is $n$-gr-flat and $L$ is Gorenstein $n$-gr-flat. We have the following pushout diagram with exact rows:
$$\xymatrix{
 & 0\ar[d] &0\ar[d] &0\ar[d] & \\
0 \ar[r] &A \ar[r]\ar@{=}[d]& B\ar[r] \ar[d] &C\ar[r] \ar[d] & 0 \\
0 \ar[r]& A \ar[r] & F^{0}\ar[r] \ar[d] &D\ar[r] \ar[d] & 0 \\
 & &L\ar@{=}[r]\ar[d] & L \ar[d]& &\\
 & &  0 &  0 & & \\
}$$
By Proposition \ref{prop1}(2), $D$ is Gorenstein $n$-gr-flat, and so we have the following commutative diagram in gr-$R$:
\begin{displaymath}\xymatrix{
   \cdots\ar[r] & P_{1} \ar[rr]\ar[dr]&& P_{0} \ar[rr]\ar[dr]&& F^{0} \ar[rr]\ar[rd]&& D^{0}\ar[rr]\ar[rd]&& D^{1}\ar[r]&\cdots,\\
     & &  L_{0}\ar[ur] \ar[rd] && A\ar[ur]\ar[rd] & &D \ar[ur] \ar[rd] & &L^{1} \ar[ur] \ar[rd]\\
      & 0\ar[ur] && 0 \ar[ur]        && 0\ar[ur]          && 0\ar[ur]   &&  0}
\end{displaymath}
where $D^i$ and $F^{0}$ are $n$-gr-flat modules, $P_{i}$ is gr-flat, $A={\rm Ker}(F^{0}\rightarrow D)$, $D={\rm Ker}(D^{0}\rightarrow D^{1})$, $L_{i}={\rm Ker}(P_{i}\rightarrow P_{i-1})$ and $L^{i}={\rm Ker}(D^{i}\rightarrow D^{i+1})$.
By Remark \ref{21}, $P_{i}$  and $L^{i}$ are Gorenstein $n$-gr-flat and hence by Lemma \ref{lem1}(2),  ${\rm Tor}_{t}^{R}(L^{i}, K_{n-1})={\rm Tor}_{t}^{R}(D, K_{n-1})=0$ for any special gr-presented module $K_{n-1}$  in $R$-gr with gr-${\rm fd}_{R}(K_{n-1})<\infty$ and any $t\geq 0$.   So, similar to the proof  $(1)\Longrightarrow(2)$ of Theorem \ref{thm1},  it follows that  $-\otimes_{R}K_{n-1}$ on the above horizontal sequence in diagram is exact and so $A$ is Gorenstein $n$-gr-flat.
  \end{proof}
  
\begin{cor}\label{cor2}
Let every $n$-presented module in $R$-gr with gr-${\rm fd}_{R}(U)<\infty$ be $(n+1)$-presented. Then,  a module $M$ in gr-$R$ is  Gorenstein $n$-gr-flat if and only if every gr-pure submodule and any gr-pure epimorphic image of $M$ are Gorenstein $n$-gr-flat.
\end{cor}
\begin{proof}
$(\Longrightarrow)$  Let $M$ be a  Gorenstein $n$-gr-flat module in gr-$R$ and $K$ a gr-pure submodule in $M$. Then, the exact sequence $0\rightarrow K\rightarrow M\rightarrow \frac{M}{K}\rightarrow 0$ is gr-pure.  So, if $K_{n-1}$ is special gr-presented module  in $R$-gr, then ${\Tor}_{1}^{R}(\frac{M}{K}, K_{n-1})=0$ and consequently  by
\cite[Lemma 2.1]{Z.JJG}, ${\Tor}_{1}^{R}(\frac{M}{K}, K_{n-1})^*\cong{\rm EXT}_{R}^{1}(K_{n-1},(\frac{M}{K})^*)=0$. Therefore, the exact sequence $0\rightarrow (\frac{M}{K})^*\rightarrow M^*\rightarrow K^*\rightarrow 0$ is special gr-pure in $R$-gr, and   using \cite[Proposition 3.10]{NG}, we deduce that  $(\frac{M}{K})^*$ is $n$-FP-gr-injective. By \cite[Proposition 3.8]{NG}, $\frac{M}{K}$ is $n$-gr-flat, and then  Proposition \ref{prop23}  shows that $K$ is Gorenstein $n$-gr-flat.

 $(\Longleftarrow )$   Let  $K$ be a gr-pure submodule in $M$. Then, the exact sequence $0\rightarrow K\rightarrow M\rightarrow \frac{M}{K}\rightarrow 0$ is gr-pure. So, it  follows, by Proposition \ref{prop1}(2), that $M$ is  Gorenstein $n$-gr-flat.
\end{proof}

Also, as for the classical injective (resp. flat) notion, the class $\mathcal{G}_{gr-\mathcal{FI}_n}$ in $R$-gr (resp. $\mathcal{G}_{gr-\mathcal{F}_n}$ in gr-$R$)  is closed under direct products (resp. direct sums).

\begin{prop}\label{prod}
Let $R$ be a graded ring. Then:
\begin{enumerate}
\item [\rm (1)]
 The class $\mathcal{G}_{gr-\mathcal{FI}_n}$ in $R$-gr
 is closed under direct products.
 \item [\rm (2)]
 The class $\mathcal{G}_{gr-\mathcal{F}_n}$ in gr-$R$
 is closed under direct sums.
 \end{enumerate}
\end{prop}
Next definition contains some general remarks about resolving classes of graded modules which will be useful  in Sections 3 and 4.  We use $gr$-$\mathscr{I}(R)$ to denote  the class of finite injective graded left modules and the symbol $gr$-$\mathscr{F}(R)$ denotes the class of finite projective graded right modules (the graded version of \cite[1.1. Resolving classes]{HH}).
\begin{Def}\label{resol1}
Let $R$ be a graded ring and $\X$ a class of graded modules. Then:
\begin{enumerate}
\item [\rm (1)]
We call $\X$ gr-injectively resolving if $gr$-$\mathscr{I}(R)\subseteq \X$, and for every short exact sequence
$0\rightarrow A\rightarrow B\rightarrow C\rightarrow0$ with $A\in\X$ the conditions $B\in\X$ and $C\in\X$ are
equivalent.
\item [\rm (2)]
We call $\X$ gr-projectively resolving if $gr$-$\mathscr{F}(R)\subseteq \X$, and for every short exact sequence
$0\rightarrow A\rightarrow B\rightarrow C\rightarrow0$ with $C\in\X$ the conditions $A\in\X$ and $B\in\X$ are
equivalent.
\end{enumerate}
\end{Def}

By Definition \ref{resol1},  Propositions \ref{prop1}, \ref{prop23}, \ref{prod}, Theorem \ref{thm1} and the graded version of \cite[Proposition 1.4]{HH}, we have the following easy observations.

\begin{prop}\label{resol2}
Assume that every $n$-presented module in $R$-gr with gr-${{\fd}_{R}(U)<\infty}$ is $(n+1)$-presented. Then:
\begin{enumerate}
\item [\rm (1)]
The class  $\mathcal{G}_{gr-\mathcal{FI}_n}$ is gr-injectively resolving.
\item [\rm (2)]
The class $\mathcal{G}_{gr-\mathcal{FI}_n}$
 is closed under direct summands.
\item [\rm (3)]
 The class $\mathcal{G}_{gr-\mathcal{F}_n}$ is gr-projectively resolving.
\item [\rm (4)]
The class $\mathcal{G}_{gr-\mathcal{F}_n}$
 is closed under direct summands.
\end{enumerate}
\end{prop}

We know that, if $R$ is a left $n$-gr-coherent ring, then every $n$-presented module in $R$-gr with gr-${{\fd}_{R}(U)<\infty}$ is $(n+1)$-presented. So in the following theorem according to previous results,  we investigate the relationships between  Gorenstein $n$-FP-gr-injective and Gorenstein $n$-gr-flat modules on $n$-gr-coherent rings.
\begin{thm}\label{thm2}
Let  $R$ be a left  $n$-gr-coherent ring. Then,
\begin{enumerate}
\item [\rm (1)]
Module $M$ in $R$-gr is  Gorenstein $n$-FP-gr-injective if and only if $M^{*}$ is Gorenstein $n$-gr-flat in gr-$R$.
\item [\rm (2)]
Module $M$ in gr-$R$ is Gorenstein $n$-gr-flat if and only if $M^{*}$ is Gorenstein $n$-FP-gr-injective in $R$-gr.
\end{enumerate}
\end{thm}
\begin{proof}
(1) $(\Longrightarrow )$ By Definition \ref{2.2}, there is an exact sequence
$\cdots\rightarrow A_1 \rightarrow A_{0} \rightarrow M \rightarrow 0 $ in $R$-gr, where every $A_i$ is $n$-FP-gr-injective, and by \cite[Theorem 3.17]{NG}, every $(A_i)^{*}$ is $n$-gr-flat in gr-$R$. So by \cite[Lemma 3.53]{Rot2}, there is an exact sequence $0\rightarrow M^*\rightarrow (A_{0})^{*}\rightarrow (A_{1})^{*}\rightarrow \cdots$  in gr-$R$. Hence, we have:
$$\X: \cdots\longrightarrow P_{1}\longrightarrow P_{0}\longrightarrow (A_{0})^{*}\longrightarrow (A_{1})^{*}\longrightarrow \cdots,$$
 where $P_i$ is gr-projective and $n$-gr-flat in gr-$R$ by Remark \ref{21} and also $M^{*}={\rm ker}((A_{0})^{*}\rightarrow (A_{1})^{*})$. Let $0\rightarrow K_{n}\rightarrow F_{n-1}\rightarrow K_{n-1}\rightarrow 0$ be a special short exact sequence in $R$-gr with gr-${\fd}_{R}(K_{n-1})<\infty$. Then,  $K_{n}$ is a gr-presented module with gr-${\fd}_{R}(K_{n})<\infty$, since $R$ is  $n$-gr-coherent. By \cite[Theorem 6.10]{Rot2} and by using the inductive presumption on gr-${\fd}_{R}(K_{n-1})$, we deduce that $\X\otimes_{R}K_{n-1}$ is exact and then $M^*$ is Gorenstein $n$-gr-flat.

$(\Longleftarrow )$ Let $M^*$ be a Gorenstein $n$-gr-flat module in gr-$R$. Then, by (2)($\Longrightarrow$),  $M^{**}$ is Gorenstein $n$-FP-gr-injective in $R$-gr. By \cite[Proposition 2.3.5]{JX}, $M$ is gr-pure in $M^{**}$, and so by Corollary \ref{cor1}, $M$ is  Gorenstein $n$-FP-gr-injective.\\
(2) ($\Longrightarrow$) By Definition \ref{2.2}, there is an exact sequence $0\rightarrow M\rightarrow F^{0}\rightarrow F^1\rightarrow\cdots$ of $n$-gr-flat modules in gr-$R$. By \cite[Proposition 3.8]{NG}, $(F^i)^{*}$ is $n$-FP-gr-injective for any $i\geq 0$. So by \cite[Lemma 3.53]{Rot2}, there is an exact sequence $\cdots\rightarrow (F^{1})^{*}\rightarrow (F^{0})^{*}\rightarrow M^* $  in $R$-gr. For a module $M^*$, there is an exact sequence $0\rightarrow M^{*}\rightarrow E_{0}\rightarrow E_{1} \rightarrow\cdots 0$ in $R$-gr, where $E_i$ is gr-injective. Consider the following exact sequence:
$$\cdots\longrightarrow (F^{1})^{*}\longrightarrow (F^{0})^{*}\longrightarrow E_{0}\longrightarrow E_{1} \longrightarrow\cdots$$
with $M^{*}={\rm ker}(E_{0}\rightarrow E_{1})$. Hence, by analogy with the proof  $(2)\Longrightarrow(1)$ of Corllary \ref{2.vc}, we obtain that  $M^*$ is Gorenstein $n$-FP-gr-injective.

($\Longleftarrow$) Let $M^*$ be a Gorenstein $n$-FP-gr-injective module in $R$-gr. Then, by (1)($\Longrightarrow$),  $M^{**}$ is Gorenstein $n$-gr-flat in gr-$R$. By \cite[Proposition 2.3.5]{JX}, $M$ is gr-pure in $M^{**}$, and so by Corollary \ref{cor2},  $M$ is Gorenstein $n$-gr-flat.
\end{proof}

Next, we are given other results of Gorenstein $n$-FP-gr-injective and $n$-gr-flat modules on $n$-gr-coherent rings.

\begin{prop}\label{limit}
Let $R$ be a left $n$-gr-coherent ring. Then,
\begin{enumerate}
\item [\rm (1)]
 the class $\mathcal{G}_{gr-\mathcal{FI}_n}$ in $R$-gr
 is closed under direct limits.
 \item [\rm (2)]
  the class $\mathcal{G}_{gr-\mathcal{F}_n}$ in gr-$R$
 is closed under direct products.
 \end{enumerate}
\end{prop}
\begin{proof}
(1) Let $U\in$ $R$-gr be an $n$-presented module and let  $\{A_i\}_{i\in I}$ be a family of  $n$-FP-gr-injective modules in $R$-gr. Then by \cite[Theorem 3.17]{NG}, $\lim_{_{_{\hspace{-.4cm}{\longrightarrow}}}}A_i$ is $n$-FP-gr-injective. So, if   $\{M_i\}_{i\in I}$ is a family of Gorenstein $n$-FP-gr-injective modules in $R$-gr, then
 the following $n$-FP-gr-injective compelex
 $${\mathbf{\Y_i}}= \cdots\longrightarrow(A_{i})_{1}\longrightarrow(A_{i})_{0}\longrightarrow (A_{i})^{0}\longrightarrow (A_{i})^{1}\longrightarrow \cdots,$$

  where $M_i={\rm ker}( (A_{i})^{0}\rightarrow  (A_{i})^{1})$, induces the following exact sequence of $n$-FP-gr-injective modules in $R$-gr:
  $$\xymatrix{\lim_{_{_{\hspace{-.4cm}{\longrightarrow}}}}{\mathbf{\Y_i}}=\cdots \longrightarrow \lim_{_{_{\hspace{-.4cm}{\longrightarrow}}}}(A_{i})_{1} \longrightarrow  \lim_{_{_{\hspace{-.4cm}{\longrightarrow}}}}(A_{i})_{0}\longrightarrow \lim_{_{_{\hspace{-.4cm}{\longrightarrow}}}}(A_{i})^{0}\longrightarrow \lim_{_{_{\hspace{-.4cm}{\longrightarrow}}}}(A_{i})^{1}\longrightarrow \cdots},$$
  where ${\lim_{_{_{\hspace{-.4cm}{\longrightarrow}}}} M_i={\rm ker}( \lim_{_{_{\hspace{-.4cm}{\longrightarrow}}}} (A_{i})^{0}\rightarrow \lim_{_{_{\hspace{-.4cm}{\longrightarrow}}}} (A_{i})^{1})}$. Assume that $K_{n-1}$ is special gr-presented module in $R$-gr with gr-${\pd}_{R}(K_{n-1})<\infty$,  then by \cite[Proposition 3.13]{NG}, $${\rm HOM}_{R}(K_{n-1},\xymatrix{\lim_{_{_{\hspace{-.4cm}{\longrightarrow}}}}{\mathbf{\Y_i}})\cong\lim_{_{_{\hspace{-.4cm}{\longrightarrow}}}}} {\rm HOM}_{R}(K_{n-1},{\mathbf{\Y_i}}).$$
  By hypothesis, ${\rm HOM}_{R}(K_{n-1},{\mathbf{\Y_i}})$ is exact, and consequently $\lim_{_{_{\hspace{-.4cm}{\longrightarrow}}}} M_i$ is Gorenstein $n$-FP-gr-injective.

  (2)
Let $U\in$ $R$-gr be an $n$-presented and  let $\{F_i\}_{i\in I}$ be a family of  $n$-gr-flat modules in gr-$R$. Then by \cite[Theorem 3.17]{NG}, $\prod_{i\in I} F_i$ is $n$-gr-flat. So, if   $\{M_i\}$ is a family of Gorenstein $n$-gr-flat modules in gr-$R$, then
 the following $n$-gr-flat compelex
 $${\mathbf{\X{_i}}}=\cdots \longrightarrow  (F_{i})_{1}\longrightarrow (F_{i})_{0}\longrightarrow (F_{i})^{0}\longrightarrow (F_{i})^{1}\longrightarrow\cdots,$$
  where $M_i={\rm ker}( (F_{i})^{0}\rightarrow  (F_{i})^{1})$, induces the following exact sequence of $n$-gr-flat modules in gr-$R$:
 $$\prod_{i\in I}{\mathbf{\X{_i}}}=\cdots \longrightarrow  \prod_{i\in I}(F_{i})_{1}\longrightarrow \prod_{i\in I}(F_{i})_{0}\longrightarrow \prod_{i\in I}(F_{i})^{0}\longrightarrow \prod_{i\in I}(F_{i})^{1}\longrightarrow\cdots,$$
  where $\prod_{i\in I} M_i={\rm ker}( \prod_{i\in I} (F_{i})^{0}\rightarrow\prod_{i\in I} (F_{i})^{1})$. If $K_{n-1}$ is special gr-presented, then $$(\prod_{i\in I}{\mathbf{\X_{i}}}\bigotimes_{R}K_{n-1})\cong\prod_{i\in I}({\mathbf{\X_{i}}}\bigotimes_{R}K_{n-1}).$$
  By hypothesis, ${\mathbf{\X{_i}}}\bigotimes_{R}K_{n-1}$ is exact, and consequently $\prod_{i\in I} M_i$ is Gorenstein $n$-gr-flat.
\end{proof}

In the following proposition, we show that if $R$ is $n$-gr-coherent, then every Gorenstein $n$-FP-gr-injective module in $R$-gr is  $n$-FP-gr-injective if $l.n$-${\rm FP}$-{\rm gr}-${\rm dim}(R)<\infty$, and every Gorenstein $n$-gr-flat module in gr-$R$ is  $n$-gr-flat if $r.n$-{\rm gr}-${\rm dim}(R)<\infty$.
\begin{prop}\label{2.2a}
Let $R$ be a left $n$-gr-coherent ring. 
\begin{enumerate}
\item [\rm (1)]
If $l.n$-${\rm FP}$-{\rm gr}-${\rm dim}(R)<\infty$, then every Gorenstein $n$-FP-gr-injective module  in $R$-gr is  $n$-FP-gr-injective.
\item [\rm (2)]
If $r.n$-{\rm gr}-${\rm dim}(R)<\infty$, then every Gorenstein $n$-gr-flat module in gr-$R$ is  $n$-gr-flat.
\end{enumerate}
\end{prop}
\begin{proof}
(1) Let $l.n$-${\rm FP}$-{\rm gr}-${\rm dim}(R)\leq k$. If $M$ is a Gorenstein $n$-FP-gr-injective module in $R$-gr, then there exists an exact sequence
$$0 \longrightarrow N\longrightarrow A_{k-1}\longrightarrow A_{k-2}\longrightarrow\cdots \longrightarrow A_{0}\longrightarrow M\longrightarrow 0$$
in $R$-gr, where every $A_{i}$ is $n$-FP-gr-injective for any $0\leq i \leq k-1$.  Since $R$ is $n$-gr-coherent for any $t\geq 1$, ${\rm EXT}_{R}^{t}( K_{n-1}, A_{i}) = 0$ for all special gr-presented left modules $K_{n-1}$ with respect to every $n$-presented module $U$ in $R$-g. Let $L_i={\ker}(A_{i}\rightarrow A_{i-1})$. Then, we have
$${\rm EXT}_{R}^{k+1}( K_{n-1},N)\cong{\rm EXT}_{R}^{k}(K_{n-1},L_{k-2})\cong\cdots\cong{\rm EXT}_{R}^{2}(K_{n-1},L_{0})\cong{\rm EXT}_{R}^{1}(K_{n-1},M).$$
Since $n$-{\rm FP}-{\rm gr}-${\rm id}_{R}(N)\leq k$, then $0={\rm EXT}_{R}^{k+1}( K_{n-1},N)\cong{\rm EXT}_{R}^{1}(K_{n-1},M)\cong{\rm EXT}_{R}^{1}(U,M)$ and consequently $M$ is $n$-FP-gr-injective.

(2) The proof is similar to that of (1).
\end{proof}

\section{Covers and Preenvelopes by Gorenstein graded Modules}
\ \ For a graded ring $R$, let $\F$ be a class of graded left $R$-modules and $M$ a graded
left $R$-module. Following \cite{MJJ, NG}, we say that a graded morphism $f : F\rightarrow M$ is an
$\F$-precover of $M$ if $F\in\F$ and ${\rm Hom}_{R-{\rm gr}}(F^{'}, F) \rightarrow {\rm Hom}_{R-{\rm gr}}(F^{'},M)\rightarrow 0$ is exact
for all $F^{'}\in\F$. Moreover, if whenever a graded morphism $g : F\rightarrow F$ such that
$fg = f$ is an automorphism of $F$, then $f : F\rightarrow M$ is called an $\F$-cover of $M$. The
class $\F$ is called (pre)covering,  if each object in $R$-gr has an $\F$-(pre)cover. Dually,
the notions of $\F$-preenvelopes, $\F$-envelopes and (pre)enveloping are defined.

In this section, by using of duality pairs on $n$-gr-coherent rings, we show that the classes  $\mathcal{G}_{gr-\mathcal{FI}_n}$ (resp. $\mathcal{G}_{gr-\mathcal{F}_n}$) or other signs are covering
and preenveloping.

\begin{Def}[The graded version of Definition 2.1 of \cite{HJ}]\label{NH}
Let  $R$ be a graded ring. Then, a duality pair over $R$ is a pair $(\M, \mathcal{C})$, where $\M$
is a class of graded left (respectively, right) $R$-modules and $\mathcal{C}$ is a class of graded
right (respectively, left) $R$-modules, subject to the following conditions:
\begin{enumerate}
\item [\rm (1)]
For any graded module $M$, one has $M\in\M$ if and only if $M^{*} \in\mathcal{C}$.
\item [\rm (2)]
$\mathcal{C}$ is closed under direct summands and finite direct sums.
\end{enumerate}
A duality pair $(\M, \mathcal{C})$ is called (co)product-closed,  if the class of $\M$ is closed
under graded direct (co)products, and a duality pair $(\M, \mathcal{C})$ is called perfect,  if it is
coproduct-closed, $\M$ is closed under extensions and $R$ belongs to $\M$.
\end{Def}

\begin{prop}\label{pairy}
Let  $R$ be a left $n$-gr-coherent ring. Then, the pair $(\mathcal{G}_{gr-\mathcal{FI}_n}, \mathcal{G}_{gr-\mathcal{F}_n})$ is a duality pair.
\end{prop}
\begin{proof}
Let $M$ be an $R$-module in $R$-gr. Then by Theorem \ref{thm2}(1), $M\in\mathcal{G}_{gr-\mathcal{FI}_n}$ if and only if $M^{*}\in\mathcal{G}_{gr-\mathcal{F}_n}$. By Proposition \ref{prod}(2), any finite direct sum of Gorenstein $n$-gr-flat modules is Gorenstein $n$-gr-flat. Also, by Proposition \ref{resol2}(4), $\mathcal{G}_{gr-\mathcal{F}_n}$ is closed under direct summands. So, by Definition \ref{NH}, the pair $(\mathcal{G}_{gr-\mathcal{FI}_n}, \mathcal{G}_{gr-\mathcal{F}_n})$ is a duality pair.
\end{proof}

\begin{prop}\label{pair}
Let  $R$ be a left $n$-gr-coherent ring. Then, the pair $(\mathcal{G}_{gr-\mathcal{F}_n}, \mathcal{G}_{gr-\mathcal{FI}_n})$ is a duality pair.
\end{prop}
\begin{proof}
Let $M$ be an $R$-module  in gr-$R$. Then by Theorem \ref{thm2}(2), $M\in\mathcal{G}_{gr-\mathcal{F}_n}$ if and only if $M^{*}\in\mathcal{G}_{gr-\mathcal{FI}_n}$. By Proposition \ref{prod}(1), any finite direct sum of Gorenstein $n$-gr-FP-injective modules is Gorenstein $n$-FP-gr-injective and  by Proposition \ref{resol2}(2), $\mathcal{G}_{gr-\mathcal{FI}_n}$ is closed under direct summands. So, by Definition \ref{NH}, the pair $(\mathcal{G}_{gr-\mathcal{F}_n}, \mathcal{G}_{gr-\mathcal{FI}_n})$ is a duality pair.
\end{proof}

\begin{thm}\label{thm3}
Let  $R$ be a left $n$-gr-coherent ring. Then:
\begin{enumerate}
\item [\rm (1)]
 The class $\mathcal{G}_{gr-\mathcal{FI}_n}$ is covering and preenveloping.
 \item [\rm (2)]
 The class $\mathcal{G}_{gr-\mathcal{F}_n}$ is covering and preenveloping.
 \end{enumerate}
\end{thm}
\begin{proof}
(1)  Every direct limit of Gorenstein $n$-FP-gr-injective modules  and every direct product  of Gorenstein $n$-FP-gr-injective modules in $R$-gr are Gorenstein $n$-FP-gr-injective by Propositions \ref{limit}(1) and \ref{prod}(1), respectively.  Also, by Corollary \ref{cor1}, the class of Gorenstein $n$-FP-gr-injective modules in $R$-gr is closed under gr-pure submodules, gr-pure quotients and gr-pure extensions. So, by Proposition \ref{pairy} and \cite[Theorem 4.2]{NG}, we deduce that  every $R$-module in $R$-gr
has a Gorenstein $n$-FP-gr-injective cover and a Gorenstein $n$-FP-gr-injective preenvelope.

(2) Every direct sum of Gorenstein $n$-gr-flat modules  and every direct product  of Gorenstein $n$-gr-flat modules in gr-$R$ are Gorenstein $n$-gr-flat by Propositions \ref{prod}(2) and \ref{limit}(2), respectively.  Also, by Corollary \ref{cor2}, the class of Gorenstein $n$-gr-flat modules in gr-$R$ is closed under gr-pure submodules, gr-pure quotients and gr-pure extensions. So, by Proposition \ref{pair} and \cite[Theorem 4.2]{NG}, we deduce that  every $R$-module in gr-$R$
has a Gorenstein $n$-gr-flat cover and a Gorenstein $n$-gr-flat preenvelope.
\end{proof}

Now we give some equivalent characterizations for $_RR$ being Gorenstein $n$-FP-gr-injective
in terms of the properties of Gorenstein $n$-FP-gr-injective and Gorenstein $n$-gr-flat modules.
\begin{thm}\label{thm4}
Let  $R$ be a left $n$-gr-coherent ring. Then,  the following statements are equivalent:
\begin{enumerate}
\item [\rm (1)]
$_RR$ is Gorenstein $n$-FP-gr-injective;
\item [\rm (2)]
Every graded module in gr-$R$ has a monic Gorenstein $n$-gr-flat preenvelope;
\item [\rm (3)]
Every gr-injective module in gr-$R$ is Gorenstein $n$-gr-flat;
\item [\rm (4)]
 Every $n$-FP-gr-injective module in gr-$R$ is Gorenstein $n$-gr-flat;
 \item [\rm (5)]
Every flat module in $R$-gr is Gorenstein $n$-FP-gr-injective;
\item [\rm (6)]
Every graded module in $R$-gr has an epic Gorenstein $n$-FP-gr-injective cover.

Moreover,  if $l.n$-${\rm FP}$-{\rm gr}-${\rm dim}(R)<\infty$, then the above conditions are also equivalent to:
\item [\rm (7)]
Every Gorenstein gr-flat module in $R$-gr is Gorenstein $n$-FP-gr-injective;
\item [\rm (8)]
Every graded module in $R$-gr is Gorenstein $n$-FP-gr-injective;
\item [\rm (9)]
Every Gorenstein gr-injective module in gr-$R$ is Gorenstein $n$-gr-flat.
\end{enumerate}
\end{thm}
\begin{proof}
$(8)\Longrightarrow (7)$, $(7)\Longrightarrow (5)$ and $(9)\Longrightarrow (3)$ are obvious.

$(1)\Longrightarrow (2)$ By Theorem \ref{thm3}(2), every module $M$ in gr-$R$ has a Gorenstein $n$-gr-flat preenvelope $f: M\rightarrow F$. By Theorem \ref{thm2}(1), $R^*$ is Gorenstein $n$-gr-flat in gr-$R$, and so $\prod_{i\in I}^{{\rm gr}-R} R^{*}$ is Gorenstein $n$-gr-flat by Proposition \ref{limit}. On the other hand, $(_RR)^*$ is a cogenerator in gr-$R$. Therefore,
exact sequence of the form  $0\rightarrow M\stackrel{\displaystyle g}\rightarrow \prod_{i\in I}^{{\rm gr}-R} R^{*}$ exists, and hence  homomorphism $ 0\rightarrow F\stackrel{\displaystyle h}\rightarrow \prod_{i\in I}^{{\rm gr}-R} R^{*}$ such that $hf=g$  shows that $f$ is monic.

$(2)\Longrightarrow (3)$ Let $E$ be a gr-injective module in gr-$R$. then $E$ has a monic Gorenstein $n$-gr-flat preenvelope $f: E\rightarrow F$ by assumption. So, the split exact sequence $0\rightarrow E\rightarrow F\rightarrow\frac{F}{E}\rightarrow 0$ exists, and so $E$ is direct summand of $F$. Hence, by Proposition \ref{resol2}, $E$  is Gorenstein $n$-gr-flat.

$(3)\Longrightarrow (1)$ By (3), $R^*$ is  Gorenstein $n$-gr-flat in gr-$R$, since $R^*$ is gr-injective. So, $R$ is Gorenstein $n$-FP-gr-injective  in $R$-gr by Theorem \ref{thm2}(1).

$(3)\Longrightarrow (4)$ Let $M$ be an $n$-FP-gr-injective module in gr-$R$ . Then by \cite[Proposition 3.10]{NG},
the exact sequence $0\rightarrow M\rightarrow E^{g}(M)\rightarrow\frac{E^{g}(M)}{M}\rightarrow 0$ is special gr-pure.
Since by (3), $E^g(M)$ is Gorenstein $n$-gr-flat, from Corollary \ref{cor2}, we deduce that $M$ is Gorenstein $n$-gr-flat.

$(4)\Longrightarrow (5)$
Let $F$ be a flat module in $R$-gr. Then, $F^*$ is gr-injective in gr-$R$, so $F^*$ is Gorenstein $n$-gr-flat by (4),
   and hence $F$ is Gorenstein $n$-FP-gr-injective by Theorem \ref{thm2}(1).

$(5)\Longrightarrow (6)$
By Theorem \ref{thm3}(1), every module $M$ in $R$-gr  has a Gorenstein $n$-FP-gr-injective cover $f:A\rightarrow M$. On the other hand, there exists an exact sequence $\bigoplus_{\gamma\in S}R(\gamma)\rightarrow M\rightarrow 0$ for some $S\subseteq G$. Since $R(\gamma)$ is Gorenstein $n$-FP-gr-injective by assumption, we have that $\bigoplus_{\gamma\in S} R(\gamma)$ is Gorenstein $n$-FP-gr-injective by Proposition \ref{limit}. Thus $f$ is an epimorphism.

$(6)\Longrightarrow (1)$ By hypothesis, $R$  has an epic Gorenstein $n$-FP-gr-injective cover $f:D\rightarrow R$,
   then we have a split exact sequence $0\rightarrow {\rm Ker}f \rightarrow D\rightarrow R\rightarrow 0$ with $D$ is a
Gorenstein $n$-FP-gr-injective module in $R$-gr. So, by Proposition \ref{resol2}, $R$ is Gorenstein $n$-FP-gr-injective in $R$-gr.

$(1)\Longrightarrow (8)$  Let $M$ be a graded left $R$-module. Then, there is an exact sequence
$\cdots\rightarrow F_1 \rightarrow F_{0} \rightarrow M \rightarrow 0 $ in $R$-gr, where each $F_i$ is gr-flat. If $R$ is a  Gorenstein $n$-FP-gr-injective module in $R$-gr, then by Proposition \ref{2.2a}(1), $R$ is  $n$-FP-gr-injective. Hence, by \cite[Theorem 4.8]{NG}, we deduce that every $F_i$ is $n$-FP-gr-injective. Also, for module $M$, there is an exact sequence $0\rightarrow M\rightarrow E_{0}\rightarrow E_{1} \rightarrow\cdots 0$ in $R$-gr, where every $E_i$ is gr-injective. So, we have:
$$\cdots\longrightarrow F_{1}\longrightarrow F_{0}\longrightarrow E_{0}\longrightarrow E_{1} \longrightarrow\cdots,$$
where $F_i$ and $E_i$ are $n$-FP-gr-injective and $M={\rm ker}(E_{0}\rightarrow E_{1})$. Thus, similar to the proof  $(2)\Longrightarrow(1)$ of Corllary \ref{2.vc}, we get that  $M$ is Gorenstein $n$-FP-gr-injective. 

$(8)\Longrightarrow (9)$ If $M$ is a Gorenstein gr-injective module in gr-$R$, then $M^*$ is in $R$-gr. So by hypothesis, $M^*$ is Gorenstein $n$-FP-gr-injective, and hence by Theorem \ref{thm2}, it follows that $M$ is Gorenstein $n$-gr-flat.
\end{proof}
\begin{ex}
  Let $R$ be a commutative, Gorenstein Noetherian, complete, local ring, $\mathfrak{m}$ its maximal ideal. Let
 $E=E(R/\mathfrak{m})$  be the $R$-injective hull of the residue field $R/\mathfrak{m}$ of $R$. By \cite[Theorem A]{Roos},
$\lambda$-$dim (R\ltimes E)= \dim R$, where $dim R$ is the Krull dimension of $R$. We suppose that $\dim R=n$, then
$(R\ltimes E)$ is $n$-gr-coherent. And if we take in \cite[Theorem 4.2]{EMH} $n=1\:\: and\:\: B=\lbrace 0\rbrace$, we get
${\rm Hom}_{R}(E,E)=R$. Then, by \cite[Corollary 4.37]{EF}, $(R\ltimes E)$ is self gr-injective which implies that $(R\ltimes E)$ is a left $n$-FP-gr-injective module over itself. Hence, $R\ltimes E$ is $n$-$FC$ graded ring ($n$-gr-coherent and $n$-FP-gr-injective), and then by Remark \ref{21}, $(R\ltimes E)$ is Gorenstein $n$-FP-gr-injective. 
  For example, the ring $R=K[[X_1,...,X_n]]$ of formal power series in $n$ variables over a field $K$ which is commutative,  Gorenstein Noetherian, complete, local ring, with $\mathfrak{m}= (X_1,...,X_n)$ its maximal ideal. We obtain
$\lambda$-$dim (R\ltimes E(R/\mathfrak{m}))=n$, that is, $R\ltimes E(R/\mathfrak{m})$ is $n$-gr-coherent ring. So according to the above
$R\ltimes E(R/\mathfrak{m})$ is $n$-$FC$ graded ring. So, every left $R\ltimes E(R/\mathfrak{m})$-module is Gorenstein
$n$-FP-gr-injective.
\end{ex}
\begin{prop}\label{pair3}
Let  $R$ be a left $n$-gr-coherent. Then,
$(\mathcal{G}_{gr-\mathcal{F}_n}, (\mathcal{G}_{gr-\mathcal{F}_n})^{\bot})$ is hereditary perfect cotorsion
pair.
\end{prop}
\begin{proof}
 Let $\mathcal{G}_{gr-\mathcal{F}_n}$ be a class of Gorenstein $n$-gr-flat modules in gr-$R$. Then, by Corolary \ref{cor2}, $\mathcal{G}_{gr-\mathcal{F}_n}$ is closed under gr-pure submodules, gr-pure
quotients and gr-pure extensions. On the other hand, $R\in\mathcal{G}_{gr-\mathcal{F}_n}$ by Remark \ref{21}, and $\mathcal{G}_{gr-\mathcal{F}_n}$ is closed under graded direct sums by Proposition \ref{prod}. So, it follows that duality pair $(\mathcal{G}_{gr-\mathcal{F}_n}, \mathcal{G}_{gr-\mathcal{FI}_n})$ is perfect. Consequently by \cite[Theorem 4.2]{NG}, $(\mathcal{G}_{gr-\mathcal{F}_n}, (\mathcal{G}_{gr-\mathcal{F}_n})^{\bot})$ is perfect cotorsion
pair. Consider the short exact sequence $0\rightarrow A\rightarrow B\rightarrow C\rightarrow 0$ in gr-$R$, where $B$ and $C$ are Gorenstein $n$-gr-flat. Then, by Proposition \ref{prop23}, $A$ is Gorenstein $n$-gr-flat and hence perfect cotorsion
pair $(\mathcal{G}_{gr-\mathcal{F}_n}, (\mathcal{G}_{gr-\mathcal{F}_n})^{\bot})$ is hereditary.
\end{proof}

\bigskip

\noindent\textbf{Acknowledgment.}  The authors
would like to thank the referee for the helpful suggestions and valuable comments.

\end{document}